\documentclass{article}
\usepackage{times}
\usepackage{amssymb, amsmath, amsthm}
\usepackage[margin=1.0in]{geometry}
\usepackage{hyperref, enumerate}
\usepackage[T1]{fontenc}

\newtheorem{theorem}{Theorem}[section]
\newtheorem{lemma}[theorem]{Lemma}
\newtheorem{proposition}[theorem]{Proposition}
\newtheorem{corollary}[theorem]{Corollary}
\theoremstyle{definition}
\newtheorem{definition}[theorem]{Definition}
\theoremstyle{remark}
\newtheorem{remark}[theorem]{Remark}

\newcommand{\Q}{\mathbb{Q}}
\newcommand{\Z}{\mathbb{Z}}
\newcommand{\C}{\mathbb{C}}
\newcommand{\N}{\mathbb{N}}
\newcommand{\BH}{\mathrm{BH}}

\newcommand{\rad}{\mathrm{rad}}
\newcommand{\z}{\zeta_{6}}

\title{On complete morphisms of Butson Hadamard matrices} 

\author{
Ronan Egan\footnotemark[1],~
Padraig \'{O} Cath\'{a}in\footnotemark[2]}
\date{
\today
}

\begin{document}

\maketitle

\begin{abstract}
A complete morphism of Butson classes replaces each entry of a matrix in $\BH(n,k)$ by an $m \times m$ block, uniformly, to produce a matrix in $\BH(mn,t)$. We show that every complete morphism is described by a pair $(U,Z)$, where $Z \in \BH(m,t)$ and $U$ is a unitary matrix annihilated by the $k^{\textrm{th}}$ cyclotomic polynomial. We deduce from a theorem of Lam and Leung that complete morphisms onto real Hadamard matrices exist precisely when $k$ is a power of $2$. The exponent-reducing morphisms of \"Osterg\aa rd--Paavola and \'O Cath\'ain--Swartz achieve the smallest block size the degree bound allows, and we show that there are no other complete morphisms for which $U$ is monomial. We construct new complete morphisms where $U$ is not monomial: a complete morphism $\BH(n,7) \rightarrow \BH(6n,6)$ is given, the first to eliminate a prime from the alphabet. 
\end{abstract}

\renewcommand{\thefootnote}{\fnsymbol{footnote}}
\footnotetext[1]{School of Mathematical Sciences, Dublin City University, Ireland. Email: \href{mailto:ronan.egan@dcu.ie}{ronan.egan@dcu.ie}}
\footnotetext[2]{\'{U}dar\'{a}s na Gaeltachta, Ireland. Email: \href{mailto:p.ocathain@gmail.com}{p.ocathain@gmail.com}}
\renewcommand{\thefootnote}{\roman{footnote}}

\section{Introduction}
A \textit{complex Hadamard matrix} is an $n\times n$ matrix $M$ with entries of modulus $1$ which satisfies $MM^{\ast} = nI_n$. Equivalently, while subject to the constraint on the modulus of entries: $M$ maximises the determinant; and the rows of $M$ are orthogonal. Let $k$ be a positive integer, and let $\zeta_{k} = e^{\frac{2\pi \sqrt{-1}}{k}}$. Denote the set of all $k^{\rm th}$ roots of unity by $\mu_{k}$. A Butson Hadamard (or Butson) matrix is a complex Hadamard matrix with entries in $\mu_k$, and the set of $n\times n$ such matrices is denoted $\BH(n,k)$. A central topic in the study of complex Hadamard matrices is the question of when $\BH(n,k)$ is non-empty: the \textit{Hadamard conjecture} asserts that $\BH(4n, 2)$ is non-empty for all $n\in \mathbb{N}$ and the \textit{complex Hadamard conjecture} asserts that $\BH(2n, 4)$ is non-empty for all $n \in \mathbb{N}$. 

Several constructions in the literature produce a matrix in $\BH(mn, t)$ from a matrix in $\BH(n,k)$ by replacing each entry with an $m \times m$ block depending only on the entry. The prototypical example is the Kronecker product of complex Hadamard matrices, which is again complex Hadamard by elementary linear algebra: if $M_j \in \BH(n_j, k_j)$ then $\bigotimes_{j \in J} M_j \in \BH(\prod_{j\in J} n_j, k)$, where $\textrm{lcm} \,( k_j \mid j \in J ) \mid k$; it is sometimes the case that $\zeta_k^i \bigotimes_{j \in J} M_j$ lies in a smaller field for suitably chosen scalar factor. Cohn showed that there exists a map from $\mu_4$ into $\BH(2,2)$, which may be applied entry-wise to construct a map $\BH(2n, 4) \rightarrow \BH(4n,2)$, \cite{Cohn1965} cf. \cite{Turyn}. The authors proposed the following: 

\begin{definition}
For arbitrary $n \in \mathbb{N}$ and fixed $k,m,t \in \mathbb{N}$, a \emph{complete morphism} $\BH(n,k) \rightarrow \BH(mn,t)$ is a function $\phi: \mu_{k} \rightarrow \mathrm{M}_{m}(\mu_{t})$ such that for every $H \in \BH(n,k)$, the blockwise image $H^{\phi} = \left( \phi(h_{ij}) \right)$ lies in $\BH(mn, t)$.
\end{definition}

A morphism is called \emph{partial} if it is not complete: an example was constructed by Compton, Craigen and de Launey with domain the \textit{unreal} matrices over sixth roots; if $H \in \BH(n, 6)$ has no entries $\pm1$ then there exists $\phi:\{ \zeta_6, \zeta_6^2, \zeta_6^4, \zeta_6^5\} \rightarrow \BH(4, 2)$ such that $H^\phi \in \BH(4n, 2)$, \cite{CCdeL}. Constructions of Gow map matrices in $\BH(n,p)$, where $p = 2^a+1$ is a Fermat prime, with no entry equal to $1$, into $\BH(2^{a}n, 4)$, \cite{GowEasy,GowHard, EO}.

The present authors developed a general framework of \textit{sound pairs} for such constructions, \cite{EO}. With Swartz, they constructed a complete morphism $\BH(n,2^s) \rightarrow \BH(2^{2^{s-1}-1}n,2)$ for each $s \geq 2$, \cite{SpectraPaper}. \"Osterg\aa rd and Paavola obtained a more general result, with  maps $\BH(n, pk) \rightarrow \BH(pn, k)$ for $p$ a prime with $p \mid k$, \cite{OP}, cf. \cite{OCS}. 

In all prior work, any morphism which eliminates a prime divisor from $k$ is incomplete, for the morphisms of Compton, Craigen and de Launey, and of Gow mentioned above which map into matrices over $\mu_{2^s}$ for some $s$, the domain restriction is essential. Our first result shows that this is forced, by way of the theorem of Lam and Leung \cite{LamLeung} on vanishing sums of roots of unity. We are led to the consideration of matrices with entries in $\mu_6$, where the Lam-Leung restriction does not hold. 

Section \ref{sec:structure} describes the structure of complete morphisms, and corollaries to that result limiting the existence of certain morphisms. In particular, this resolves the spectrum of existence for complete morphisms $\BH(n,k) \rightarrow \BH(mn,2)$ for all $k \in \mathbb{N}$. In Section \ref{sec:morphisms} we consider the complete morphisms with the smallest parameters not eliminated by the obstruction, and provide a new example of a complete morphism $\BH(n,7) \rightarrow \BH(6n,6)$. This is the first example of a complete morphism which eliminates a prime divisor of $k$ from the alphabet. 

\section{The structure of complete morphisms}\label{sec:structure}
Throughout this paper, $\Phi_{k}$ is the $k^{\textrm{th}}$ cyclotomic polynomial. We write $\rad(k)$ for the product of the distinct primes dividing $k$. We begin with some straightforward but powerful consequences of the definition of a complete morphism. 

\begin{lemma}\label{lem:relations}
Suppose that $\phi$ is a complete morphism $\BH(n,k) \rightarrow \BH(mn,t)$. 
For $\zeta^j \in \mu_{k}$ write $Z_{j} = \phi(\zeta^j)$. Let $p$ be a prime dividing $k$ 
and set $e = k/p$. Indices are read modulo $k$ throughout. Then for every $j \in \{0,1,\ldots,k-1\}$
\begin{enumerate}[(1)]
\item $Z_{j} \in \BH(m,t)$;
\item $\sum_{a=0}^{p-1} Z_{j +ea} = 0$;
\item $Z_{j+a}Z_{a}^{\ast}$ is independent of $a$.
\end{enumerate}
\end{lemma}

\begin{proof}
Up to Hadamard equivalence, $H \in \BH(n,k)$ contains a constant row with all entries equal to $\zeta^j$. Since $H^\phi\in \BH(mn,t)$, clearly all entries of $Z_j$ are in $\mu_t$ and $n Z_jZ_j^\ast = mn I_m$, and hence $Z_j \in \BH(m,t)$.

Since $p \mid k$, the Fourier matrix $F_p = (\zeta_{p}^{ij})_{0 \leq i,j < p}$ belongs to $\BH(p, k)$. Orthogonality of the first two rows of $(\zeta_{k}^{j}F_{p})^{\phi}$ requires that $\bigl(\sum_{a=0}^{p-1} Z_{j+ea}\bigr) Z_{j}^{\ast} = 0_{m\times m}$. Since $Z_{j}$ is invertible by (1), the claim follows.

For (3), fix $a$ and let $F'$ be obtained from $F_{k} = (\zeta^{rc})_{0 \leq r,c < k}$ by multiplying the last column by $\zeta^{a}$; scaling a column by a root of unity preserves $\BH(k,k)$. Orthogonality of block rows $0$ and $r$ of $(F')^{\phi}$ gives
\[
\sum_{c=0}^{k-2} Z_{rc} Z_{0}^{\ast} \;+\; Z_{r(k-1)+a}Z_{a}^{\ast} \;=\; 0 .
\]
The first summand does not involve $a$, so $Z_{r(k-1)+a}Z_{a}^{\ast}$ is independent of $a$. As $r$ runs over $1, \ldots, k-1$ the residue $r(k-1) \equiv -r$ runs over all non-zero residues modulo $k$; writing $j$ for $-r$ gives the result.
\end{proof}

\begin{theorem}\label{thm:structure}
There is a bijection between complete morphisms $\BH(n,k) \rightarrow \BH(mn,t)$ and 
pairs $(U, Z)$ where $Z \in \BH(m,t)$ and $U$ is unitary with $\Phi_{k}(U) = 0$ such that $U^{j}Z$ has entries in $\mu_{t}$ for $0 \leq j < k$. 
\end{theorem}

\begin{proof} 
First assume that $\phi$ is a complete morphism and set $Z = Z_{0}$ and let $\rho(\zeta^{j}) = \frac{1}{m}Z_{j}Z^{\ast}$. Then $\rho(\zeta_{k}^{j})\rho(\zeta_{k}^{j})^{\ast} = I_{m}$, and by Lemma \ref{lem:relations}(3), $\rho(\zeta_{k}^{i}\zeta_{k}^{j}) = \rho(\zeta_{k}^{i})\rho(\zeta_{k}^{j})$, so $\zeta_{k}^{j} \mapsto \rho(\zeta_{k}^{j})$ is a homomorphism into the unitary group and $Z_{j} = \rho(\zeta_{k}^{j})Z$. Assign $U = \rho(\zeta_{k})$; then $U^{k} = I$ and Lemma \ref{lem:relations}(2) gives $\sum_{i<p} U^{ik/p} = 0_{m \times m}$ for every prime $p \mid k$, so every eigenvalue of $U$ is a primitive $k^{\textrm{th}}$ root of unity and, $U$ being unitary, $\Phi_{k}(U) = 0$. Finally $U^{j}Z = Z_{j}$ has all of its entries in $\mu_{t}$ for every $j$, by Lemma \ref{lem:relations}(1).

Conversely, given $(U,Z)$, let $H \in \BH(n,k)$; the entries of $H^{\phi}$ lie in $\mu_{t}$ because those of every $U^{j}Z$ do. Let $r, s$ be distinct rows of $H$ with $\sum_{j} r_{j}\overline{s_{j}} = 0$. Write $r_{j} = \zeta_{k}^{a_{j}}$ and $s_{j} = \zeta_{k}^{b_{j}}$, and set $e_{j} = a_{j} - b_{j} \bmod k$. The polynomial $P(x) = \sum_{j} x^{e_{j}} \in \Z[x]$ vanishes at $\zeta_{k}$, hence is divisible by $\Phi_{k}$, hence $P(U) = 0$. Since $ZZ^{\ast} = mI$ and $U^{k} = I$, the inner product of the corresponding block-rows of $H^{\phi}$ is $\sum_{j} U^{a_{j}}ZZ^{\ast}U^{-b_{j}} = m P(U) = 0$, while each diagonal block of the Gram matrix is $mnI_{m}$; so $H^{\phi} \in \BH(mn,t)$.
\end{proof} 

\begin{remark}
The bijection of Theorem \ref{thm:structure} is entirely explicit: given a morphism $\BH(n,k) \rightarrow \BH(mn,t)$, recover $Z = \phi(1)$ and then $U = \phi(\zeta_{k})Z^{\ast}$; this inverts $(U,Z) \mapsto \phi$, since $\phi(\zeta_{k}^{j}) = U^{j}Z$.
\end{remark}

It is natural to give a name to the pairs of matrices $(U,Z)$ satisfying the hypotheses of Theorem \ref{thm:structure}.

\begin{definition}\label{def:pair}
A pair $(U,Z)$ of $m \times m$ complex matrices is $(k,t)$-\emph{flat} (or just \emph{flat} when the parameters $(k,t)$ are assumed) if $U$ is unitary, $\Phi_{k}(U) = 0$, and every $U^{j}Z$ has all of its entries in $\mu_{t}$ for all $0 \leq j < k$. A flat pair $(U,Z)$ is a $(k,t)$-\emph{amiable pair} (or just amiable pair) if moreover $ZZ^{\ast} = mI_{m}$.
\end{definition}

\subsection{Integrality conditions}

\begin{proposition}\label{prop:integrality}
Let $(U,Z)$ be a $(k,t)$-amiable pair of $m \times m$ matrices. Then $mU$ has entries in $\Z[\zeta_{t}]$, and $d \mid m$ for $d = \varphi(\mathrm{lcm}(k,t))/\varphi(t)$.
\end{proposition}

\begin{proof}
Since $ZZ^{\ast} = mI_{m}$ we have $mU = (UZ)Z^{\ast}$, and both factors are flat, so the entries of $mU$ lie in $\Z[\zeta_{t}]$. Hence the characteristic polynomial of $U$ lies in $\Q(\zeta_{t})[x]$, and as $\Phi_{k}(U) = 0$ it is a product of irreducible factors of $\Phi_{k}$ over $\Q(\zeta_{t})$. Those factors correspond to the orbits of $G = \mathrm{Gal}(\Q(\zeta_{t})(\zeta_{k})/\Q(\zeta_{t}))$ on the primitive $k^{\textrm{th}}$ roots of unity, and $G$ acts freely, so each has degree $d = |G| = \varphi(\mathrm{lcm}(k,t))/\varphi(t)$; the characteristic polynomial has degree $m$, so $d \mid m$.
\end{proof}

\subsection{Necessary conditions for existence of complete morphisms}

\begin{corollary}\label{thm:obstruction}
Suppose there exists a complete morphism $\BH(n,k) \rightarrow \BH(mn,t)$. Then every prime divisor of $k$ is a non-negative integer combination of the prime divisors of $t$.
\end{corollary}

\begin{proof}
Let $p \mid k$ and set $e = k/p$. By Lemma \ref{lem:relations}(2), $\sum_{a=0}^{p-1} Z_{ea} = 0_{m \times m}$. The entries of each $Z_{ea}$ are elements of $\mu_{t}$; hence any component of the matrix equation yields a vanishing sum of exactly $p$ elements of $\mu_{t}$.

By a well-known theorem of Lam and Leung, if there exists a vanishing sum with elements in $\mu_t$ then the number of summands is in the set $\N p_{1} + \cdots + \N p_{r}$ where $p_{1}, \ldots, p_{r}$ are the primes dividing $t$, \cite{LamLeung}. 
\end{proof}

Combinations of a single prime are especially restrictive:

\begin{corollary}\label{cor:primepower}
Let $t = q^{b}$ be a prime power. A complete morphism $\BH(n,k) \rightarrow \BH(mn, q^{b})$ exists if and only if $k$ is a power of $q$. In particular, complete morphisms $\BH(n,k) \rightarrow \BH(mn,2)$ exist if and only if $k = 2^{\alpha}$.
\end{corollary}

\begin{proof}
Every non-negative combination of the single prime $q$ is a multiple of $q$, so Corollary \ref{thm:obstruction} forces $\rad(k) = q$. Conversely, for $k = q^{a}$ with $a \leq b$ the inclusion $\mu_{k} \subseteq \mu_{t}$ gives a morphism with $m = 1$, and for $a > b$ the \"Osterg\aa rd--Paavola
and \'{O} Cath\'{a}in--Swartz constructions \cite{OP, OCS} gives one with $m = q^{a-b}$, which is optimal: by Proposition \ref{prop:integrality}, $m$ is divisible by $d = \varphi(q^{a})/\varphi(q^{b}) = q^{a-b}$.
\end{proof}

For real targets this settles the question raised at the end of \cite{EO}: the complete morphisms $\BH(n,k) \rightarrow \BH(mn,2)$ are exactly those with $k = 2^{\alpha}$, realised with the minimal multiplier $m = 2^{\alpha - 1}$ by \cite{OP,OCS}, and every morphism from an alphabet containing an odd prime onto real Hadamard matrices is necessarily partial.

The complete morphisms constructed to date by the authors and Swartz and by \"Osterg\aa rd--Paavola correspond to pairs $(U, Z)$ where $U$ is monomial; as a result flatness of the powers $U^{j}Z$ is automatic there, \cite{OCS, OP}. 

\begin{proposition}\label{prop:monomial}
Fix $k$ and $t$. There is a $(k,t)$-amiable pair $(U,Z)$ with $U$ monomial if and only if $\rad(k) \mid t$.
\end{proposition}

\begin{proof}
If $\rad(k) \mid t$, the matrices of \cite{OCS} are monomial. Conversely, let $(U,Z)$ be amiable with $U$ monomial: its phases lie in $\mu_{t}$, being quotients of entries of $UZ$ and of $Z$. Suppose $p^{a} \| k$ with $p \nmid t$. Then $W = U^{k/p^{a}}$ is again monomial with entries in $\mu_{t}$, and every eigenvalue of $W$ is a primitive $p^{a^\textrm{th}}$ root of unity. The eigenvalues of $W$ arising from a $c$-cycle of its permutation part with phase product $\delta \in \mu_{t}$ form the full solution set of $\omega^{c} = \delta$. Each solution lies in $\mu_{p^{a}}$, so $\delta = \omega^{c} \in \mu_{p^{a}} \cap \mu_{t} = \{1\}$; but then $\omega = 1$ is among the solutions, and $1$ is not a primitive $p^{a^\textrm{th}}$ root of unity.
\end{proof}

Corollary \ref{cor:primepower} and Proposition \ref{prop:monomial} together establish the limits of all previous constructions of complete morphisms. We observe that these results generalise to partial morphisms: the domain of the Compton-Craigen-de Launey partial morphism is maximal, as adding another root of unity to the domain would give a vanishing sum of length $3$ in $\{ \pm 1\}$; and similarly for the unreal matrices constructed from matrices of Gow. In the preparation of this note, the authors became aware of a minor inaccuracy in previous work, which is recorded below for the sake of correcting the record; no results of the previous paper are affected.

\begin{remark}[Erratum to \cite{EO}]\label{rem:erratum}
The proof of Theorem 10 of \cite{EO}, which realises Gow's matrices $M'_{p}$ for a Fermat prime $p = 2^{a}+1$, states that the trace of $2^{-a}M'_{p}$ is $-p$. The correct value is $-(p-1) = -2^{a}$, and the conclusion of the theorem is unaffected. 
\end{remark}

\section{The existence of a complete morphism $\BH(n,7) \rightarrow \BH(6n,6)$}\label{sec:morphisms}

In this section we exhibit a $(7,6)$-amiable pair $(U,Z)$ with $m = 6$, which is the smallest block size Proposition \ref{prop:integrality} permits, since $d = \varphi(42)/\varphi(6) = 6$. The search that produced it enumerates flat orbits rather than matrices $U$, and is described in \S\ref{subsec:search}. 

\begin{theorem}\label{thm:main}
There is a complete morphism $\BH(n,7) \rightarrow \BH(6n,6)$. 
\end{theorem}

The matrices displayed below prove the theorem, and the enumeration of \S\ref{subsec:search} shows that up to monomial equivalence exactly four matrices $U$ occur in a $(7,6)$-amiable pair at block size $m = 6$, falling into two families of two, the classes within each family interchanged by $U \mapsto U^{a}$ for $a$ a non-residue modulo $7$. A single $Z$ serves for both: put $Z = (\zeta_{3}^{z_{ij}})$, 
\[
\setlength{\arraycolsep}{2.5pt}
2U_{1} \;=\;
\begin{pmatrix}
 0 & 0 & 0 & 0 & 0 & 2 \\
 1 & \z^{2} & 1 & \z^{5} & 0 & 0 \\
 0 & 1 & \z^{2} & \z^{5} & 1 & 0 \\
 \z^{3} & \z^{3} & 0 & \z^{4} & \z & 0 \\
 \z^{4} & \z^{2} & \z^{2} & 0 & \z^{4} & 0 \\
 \z^{4} & 0 & 1 & \z & 1 & 0
\end{pmatrix},\;
3U_{2} \;=\;
\begin{pmatrix}
 1 & 1 & 1 & \z^{4} & 2\z & \z^{4} \\
 \z^{2} & \z^{2} & 2\z^{5} & 1 & 1 & 1 \\
 \z^{2} & \z^{2} & \z^{2} & 2\z^{3} & 1 & 1 \\
 1 & 2\z^{3} & 1 & \z^{4} & \z^{4} & \z^{4} \\
 2\z^{3} & 1 & 1 & \z^{4} & \z^{4} & \z^{4} \\
 \z^{2} & \z^{2} & \z^{2} & 1 & 1 & 2\z^{3}
\end{pmatrix},\;
(z_{ij}) \;=\;
\begin{pmatrix}
 0 & 0 & 0 & 0 & 0 & 0 \\
 0 & 0 & 1 & 1 & 2 & 2 \\
 0 & 1 & 0 & 2 & 1 & 2 \\
 0 & 2 & 1 & 2 & 0 & 1 \\
 0 & 2 & 2 & 1 & 1 & 0 \\
 0 & 1 & 2 & 0 & 2 & 1
\end{pmatrix}.
\]
Then $U_{1}$ and $U_{2}$ are unitary; $\Phi_{7}(U_{i}) = 0$; $Z \in \BH(6,6)$; and $U_{i}^{j}Z \in \BH(6,6)$ for $0 \leq j < 7$. By Theorem \ref{thm:structure} each assignment $\phi_{i}(\zeta_{7}^{j}) = U_{i}^{j}Z$ is a complete morphism $\BH(n,7) \rightarrow \BH(6n,6)$.

Let $(U,Z)$ be $(k,t)$-amiable and let $P$, $Q$ be monomial over $\mu_{t}$. Then $(PUP^{-1}, PZQ)$ is amiable too, since $P$ and $Q$ are unitary and $(PUP^{-1})^{j}PZQ = P(U^{j}Z)Q$. The right factor $Q$ acts on $Z$ without touching $U$, and it is this freedom that allows a single $Z$ to serve both pairs above.

\subsection{Flat orbits}

Let $p$ be prime and let $(U,Z)$ be a $(p,t)$-amiable pair with $p \nmid t$. Each column $c$ of $Z$ generates the sequence $c_{j} = U^{j}c$, $0 \leq j \leq p-1$; these are columns of the matrices $U^{j}Z$ and so lie in $\mu_{t}^{m}$. Call such a sequence a \emph{flat orbit} and $c$ a \emph{flat orbit generator}. We will show that flat orbits are characterised without reference to $U$.

\begin{lemma}\label{lem:orbit}
Let $c_{0}, \ldots, c_{p-1} \in \mu_{t}^{m}$ with $m \leq p-1$, and suppose $c_{0}, \ldots, c_{m-1}$ are linearly independent. Then there is a unitary $m \times m$ matrix $U$ with $U^{j}c_{0} = c_{j}$ for all $j$, $U^{p} = I$ and $\Phi_{p}(U) = 0$, if and only if
\begin{enumerate}[(1)]
\item $\langle c_{i}, c_{j}\rangle$ depends only on $j - i$ modulo $p$, and
\item $\sum_{j} c_{j} = 0$.
\end{enumerate}

\end{lemma}

\begin{proof}
Suppose the matrix $U$ with the stated properties exists. Then $\langle c_{i}, c_{j}\rangle = \langle U^{i}c, U^{j}c\rangle = \langle c, U^{j-i}c\rangle$ as $U$ is unitary, implying part (1). That $\sum_{j} U^{j} = 0$ follows because $\Phi_{p}(x) = \sum_{j<p} x^{j}$ for $p$ prime and by assumption, $\Phi_{p}(U) = 0$. Hence $\sum_{j} c_{j} = \bigl(\sum_{j}U^{j}\bigr)c_{0} = 0$, implying part (2).

Conversely, assume claims (1) and (2) hold and write $\gamma_{r} = \langle c_{0}, c_{r}\rangle$, so that by (1) any $m$ consecutive $c_{j}$ have Gram matrix $\Gamma = (\gamma_{j-i})_{0 \leq i,j < m}$. Let $M_{1}$ and $M_{2}$ be the matrices with columns $(c_{0}, \ldots, c_{m-1})$ and $(c_{1}, \ldots, c_{m})$; both have Gram matrix $\Gamma$, and $M_{1}$ is invertible by hypothesis, so $\Gamma$ is too. Let $U = M_{2}M_{1}^{-1}$, the unique linear map with $Uc_{j} = c_{j+1}$ for $0 \leq j < m$. Then
\[
M_{1}^{\ast}U^{\ast}UM_{1} \;=\; (UM_{1})^{\ast}(UM_{1}) \;=\; M_{2}^{\ast}M_{2} \;=\; \Gamma \;=\; M_{1}^{\ast}M_{1},
\]
so $U^{\ast}U = I$ and $U$ is unitary.

We claim that $Uc_{l} = c_{l+1}$ for every $l < p-1$. Write $W_{l} = [\,c_{l-m+1} \mid \cdots \mid c_{l}\,]$, so that $W_{m-1} = M_{1}$ and $W_{m} = M_{2}$, and induct on $l \geq m$. If $Uc_{j} = c_{j+1}$ for all $j < l$ then $UW_{l-1} = W_{l}$, whence $W_{l}^{\ast}W_{l} = W_{l-1}^{\ast}W_{l-1} = \Gamma$ and $W_{l}$ is invertible. For $l - m < i \leq l$,
\[
\langle c_{i}, Uc_{l}\rangle \;=\; \langle Uc_{i-1}, Uc_{l}\rangle \;=\; \langle c_{i-1}, c_{l}\rangle \;=\; \gamma_{l-i+1} \;=\; \langle c_{i}, c_{l+1}\rangle ,
\]
that is $W_{l}^{\ast}(Uc_{l}) = W_{l}^{\ast}c_{l+1}$, and $Uc_{l} = c_{l+1}$ follows.

By claim (2),
\[
Uc_{p-1} \;=\; U\Bigl(-\!\!\sum_{j<p-1}\! c_{j}\Bigr) \;=\; -\!\!\sum_{j<p-1}\! c_{j+1} \;=\; -\!\!\sum_{1 \leq j \leq p-1}\! c_{j} \;=\; c_{0}, \]
so $U$ permutes the $c_{j}$ in a $p$-cycle and $U^{p} = I$. Finally $\bigl(\sum_{j}U^{j}\bigr)c_{0} = \sum_{j} c_{j} = 0$; the kernel of $\sum_{j}U^{j}$ is $U$-invariant and contains $c_{0}$, hence contains $\mathrm{span}\{U^{j}c_{0}\} = \C^{m}$. So $\sum_{j}U^{j} = 0$, that is $\Phi_{p}(U) = 0$.
\end{proof}

\begin{remark}
The matrix $U$ of Lemma \ref{lem:orbit} is unique and determined explicitly from $c_{0},\ldots,c_{m}$ by
\begin{equation}\label{eq:solve}
U \;=\; M_{2}M_{1}^{-1},
\end{equation}
where $M_{1}$ is invertible by the assumption of linear independence of $\{c_{0},\ldots,c_{m-1}\}$.
\end{remark}

\begin{lemma}\label{lem:nondegenerate}
Let $(U,Z)$ be a $(p,t)$-amiable pair where $p \nmid t$ with $m = p-1$, and let $c$ be any column of $Z$. Then $\{c, Uc, \ldots, U^{p-2}c\}$ is linearly independent.
\end{lemma}

\begin{proof}
Put $K = \Q(\zeta_{t})$ and $L = K(\zeta_{p})$, and let $G = \mathrm{Gal}(L/K)$, which is cyclic and acts simply transitively on the primitive $p^{\textrm{th}}$ roots of unity, since $[L:K] = \varphi(pt)/\varphi(t) = p-1$ for $p \nmid t$. The entries of $U$ lie in $K$, because $mU = (UZ)Z^{\ast}$ has entries in $\Z[\zeta_{t}]$, and $c \in \mu_{t}^{m} \subseteq K^{m}$.

The characteristic polynomial of $U$ lies in $K[x]$, and its roots are primitive $p^{\textrm{th}}$ roots of unity because $\Phi_{p}(U) = 0$. Its multiset of roots is therefore stable under $G$, hence a union of $G$-orbits; each orbit is the full set of $p-1$ primitive $p^{\textrm{th}}$ roots, and the degree is $m = p-1$, so each occurs exactly once. Choose $x_{\lambda}$ from each of the $p-1$ distinct $1$-dimensional $\lambda$-eigenspaces $E_{\lambda}$ so that $c = \sum_{\lambda}x_{\lambda}$.

For $\sigma \in G$, applying $\sigma$ entrywise to $(U - \lambda)x_{\lambda} = 0$ and using $\sigma(U) = U$ gives $\sigma(x_{\lambda}) \in E_{\sigma(\lambda)}$; since $\sigma(c) = c$, uniqueness of the decomposition yields $x_{\sigma(\lambda)} = \sigma(x_{\lambda})$. Hence the $x_{\lambda}$ vanish either all together or not at all, and $c \neq 0$ forces all of them to be non-zero. So the cyclic subspace generated by $c$ is all of $\C^{m}$, and its spanning set $\{c, Uc, \ldots, U^{m-1}c\}$ is a basis.
\end{proof}

\subsection{The search for a complete morphism}\label{subsec:search}

The classification is by exhaustive enumeration of flat orbits, in two stages. 

\emph{Stage 1: enumerating flat orbits.} Fix $p = 7$, $t = 6$, $m = 6$. An amiable pair $(U,Z)$ may be replaced by $(PUP^{-1}, PZ)$ for any monomial $P$ over $\mu_{6}$, so a chosen column orbit may be normalised: taking $P = \mathrm{diag}(\overline{c_{0}})$ makes $c_{0} = \mathbf{1}$, the column of all ones, and the residual freedom is the symmetric group permuting the six coordinates, which fixes $c_{0}$. This allows us to restrict to columns $c_{1}$ with non-decreasing exponents in the entries, hence there are $\binom{11}{5} = 462$ candidates for $c_{1}$. Every amiable pair is therefore reached: normalising the orbit of the first column of $Z$ in this way gives a sequence satisfying Lemma \ref{lem:orbit}(1) and (2), so it is enumerated, and its first six terms are independent by Lemma \ref{lem:nondegenerate}, so \eqref{eq:solve} returns the conjugated $U$. 

Lemma \ref{lem:orbit}(1) requires that $\langle c_i, c_j \rangle = \gamma_{j-i}$ depends only on $j-i$; hence each additional column added to a flat orbit imposes a linear constraint on subsequent columns. Since $\gamma_{-r} = \overline{\gamma_{r}}$, all inner products are determined once $c_{1}, c_{2}, c_{3}$ are chosen; the last column is not searched, being forced as $c_{6} = -\sum_{j<6}c_{j}$ by Lemma \ref{lem:orbit}(2) and retained only if it is flat. The enumeration yields $30{,}728$ flat orbits in seconds. 

Each flat orbit determines $U$ uniquely via equation \eqref{eq:solve}. The flat orbit is retained if and only if $6U \in \mathrm{M}_{6}(\Z[\z])$, which is necessary by Proposition \ref{prop:integrality}. Of these, $3{,}560$ yield an integral $6U$. 

\emph{Stage 2: constructing $Z$ from $U$.} Whether a column generates a flat orbit of $U$ depends on that column alone, so the flat orbit generators of a fixed $U$ are found by one sweep of the $6^{6}$ elements of $\mu_{6}^{6}$. A matrix $Z$ completing $U$ to an amiable pair is exactly a set of six pairwise orthogonal such columns, each of norm $\sqrt{6}$; so Stage 2 is a search for $6$-cliques in the orthogonality graph on the flat orbit generators of $U$. It is run for every $U$ retained at Stage 1, and the resulting amiable pairs are reduced by monomial equivalence, leaving four classes.

\begin{remark}[Cost of the search]\label{rem:cost}
A complete morphism $\BH(n,k) \rightarrow \BH(mn,t)$ exists for some $m$ if and only if, for each prime $p \mid k$ with $p \nmid t$, a complete morphism $\BH(n,p) \rightarrow \BH(m_{p}n,t)$ exists for some $m_{p}$. By Proposition~\ref{prop:integrality}, $(p-1) \mid m_{p}$; the search treats $m_{p} = p-1$, the only case in which a single flat orbit determines $U$.

In Stage~1 each partial orbit $(c_{1},\ldots,c_{j-1})$ is expanded by scanning the $v \in \mu_{t}^{p-1}$ with $\langle c_{j-1}, v\rangle = \gamma_{1}$, each tested against at most $p$ inner products. There are at most $t^{p-1}$ such $v$ and, since $\langle \mathbf{1}, v\rangle$ takes at most $\binom{p+t-2}{p-1}$ values, at least $t^{p-1}/\binom{p+t-2}{p-1}$ for some choice of $c_{1}$. Up to factors polynomial in $p$, Stage~1 therefore costs between $t^{2(p-1)}$ (expanding depth~$2$) and $t^{(p-1)(p-3)}$ (the full tree) operations. Stage~2 multiplies each vector of $\mu_{t}^{p-1}$ by $U$ a total of $p-1$ times, at a cost of $\mathcal{O}(p^{3}t^{p-1})$ operations per $U$.

We ran the search in a NumPy implementation on a laptop; the table gives the number of nodes in the Stage~1 tree and of candidates scanned. Both are exact for $(5,6)$ and $(7,6)$. Elsewhere they are estimated by Knuth's random-probe method~\cite{Knuth75}, which reproduces both counts at $(7,6)$ to within one standard error. Times are projected at about $3 \times 10^{7}$ candidates per second, which is the rate measured at $(7,6)$.
\begin{center}
\begin{tabular}{ccrrr}
$(p,t)$ & $m$ & nodes & candidates & Stage~1 time \\ \hline
$(5,6)$ & $4$ & $351$ & $1.5 \times 10^{4}$ & $0.2$\,s \\
$(7,6)$ & $6$ & $1.1 \times 10^{6}$ & $1.3 \times 10^{9}$ & $38$\,s \\
$(7,12)$ & $6$ & $6.3 \times 10^{7}$ & $5.7 \times 10^{11}$ & $5$\,h \\
$(11,6)$ & $10$ & $5 \times 10^{15}$ & $> 1.3 \times 10^{15}$ & $> 1$\,year
\end{tabular}
\end{center}
At $(5,6)$ there is no flat orbit at all, so no complete morphism $\BH(n,5) \rightarrow \BH(4n,6)$ exists. Stage~2 took about $30$\,ms per $U$ at $(7,6)$, under two minutes for all $3{,}560$ matrices.  and its sweep of $6^{10}$ vectors at $(11,6)$ would take minutes per $U$; asymptotically, Stage~1 dominates the running time.
\end{remark}

\section*{Statement of AI Use}

The authors used Anthropic's \textit{Claude Fable} and \textit{Claude Opus 3.5} to generate Python scripts to search for morphisms. Initial versions of some proofs were also generated with Claude. All proofs were rederived and rewritten by the authors; and computational results were checked and verified in \textsc{MAGMA}.  

\bibliographystyle{abbrv}
\flushleft{
\bibliography{MyBiblio}
}

\end{document}